\documentclass[a4paper,12pt]{article}

\usepackage[intlimits]{amsmath}
\usepackage{amssymb}
\usepackage{amsthm}
\usepackage{mathrsfs}
\usepackage{cite}
\usepackage{color}
\usepackage{bm}

\newcommand{\la}{\langle}
\newcommand{\ra}{\rangle}
\newcommand{\pr}{\partial}
\newcommand{\dom}{\Omega}
\newcommand{\sig}{\Sigma}
\newcommand{\gam}{\Gamma}
\newcommand{\re}{\mathfrak{Re}\;}
\newcommand{\im}{\mathfrak{Im}\;}

\newcommand{\N}{{\mathcal{N}}}

\newcommand{\R}{\mathbb{R}}
\newcommand{\C}{\mathbb{C}}

\newcommand{\Ord}{\mathscr{O}}

\newcommand{\dd}{\,\text{d}}
\newcommand{\supp}{\mbox{supp\;}}

\newtheorem{thm}{Theorem}
\newtheorem{lem}{Lemma}
\newtheorem{dfn}{Definition}

\title{An inverse problem for the\\ porous medium equation\\ with partial data and a possibly singular absorption term}
\author{C\u{a}t\u{a}lin I. C\^{a}rstea\thanks{School of Mathematics, Sichuan University, Chengdu, Sichuan, 610064, P.R.China; \mbox{email: catalin.carstea@gmail.com}}
\and Tuhin Ghosh\thanks{Department of Mathematics, Universit\"at Bielefeld, 33615 Bielefeld, Germany; \mbox{email: tghosh@math.uni-bielefeld.de}}
\and Gunther Uhlmann\thanks{Department of Mathematics, University of Washington, Seattle, WA 98195-4350, USA,
        and Institute for Advanced Study of the Hong Kong University of Science and Technology, Hong Kong; \mbox{email: gunther@math.washington.edu}}
}
\date{}

\begin{document}
\maketitle

\begin{abstract}
In this paper we prove uniqueness in the inverse boundary value problem for the three coefficient functions in the porous medium equation with an absorption term $\epsilon\pr_t u-\nabla\cdot(\gamma\nabla u^m)+\lambda u^q=0$, with $m>1$, $m^{-1}<q<\sqrt{m}$, with the space dimension 2 or higher. This is a degenerate parabolic type quasilinear PDE which has been used as a model for phenomena in fields such as gas flow (through a porous medium), plasma physics, and population dynamics. In the case when $\gamma=1$ a priori, we prove unique identifiability with data supported in an arbitrarily small part of the boundary. Even for the global problem we improve on previous work by obtaining uniqueness with a finite (rather than infinite) time of observation and also by introducing the additional absorption term $\lambda u^q$. 
\end{abstract}


\section{Introduction}

Let $\dom\subset\R^n$, with $n\geq 2$, be a bounded, smooth domain, and let $T\in(0,\infty)$. We would like to consider the following equation in the space-time cylindrical domain $(0,T)\times\dom$
\begin{equation}\label{eq}
\left\{\begin{array}{l}\epsilon(x)\pr_t u(t,x) -\nabla\cdot(\gamma(x)\nabla u^m(t,x))+\lambda(x) u^q(t,x)=f(t,x),\\[5pt] u(0,x)=0,\quad u|_{[0,T)\times\pr\dom}=\phi(t,x),\quad u\geq0.\end{array}\right.
\end{equation}
In the above $m$ and $q$ are fixed parameters that we assume to satisfy 
\begin{equation}
 m>1,\quad m^{-1}<q<\sqrt{m}.
\end{equation}
 The coefficients $\epsilon$, $\gamma$, and $\lambda$ are bounded nonnegative functions. $\epsilon$ and $\gamma$ will be assumed to also have strictly positive lower bounds. We will also assume that $\epsilon,\gamma, \lambda\in C^\infty(\overline{\dom})$, for the sake of convenience. Consistency requires that $\phi(0,x)=0$, $\phi\geq 0$.

The equation \eqref{eq} is usually refered to as a porous medium equation, with an absorbtion term (i.e. the $\lambda u^q$ term). The name comes from its use in modeling the flow of a gas through a porous medium (see \cite{Sch},\cite{V}), but equations of this form are also used to model phenomena in other fields, such as plasma physics (see \cite{rosenau1983thermal}), or population dynamics (see \cite{namba1980density}). 

From the mathematical point of view, equation \eqref{eq} is a degenerate parabolic type quasilinear equation. Note that if $0<q<1$, then the absorption term $\lambda u^q$ is not Lipschitz in $u$ and  it is said to be singular in this case. Furthermore, equation \eqref{eq} is a particularly simple modification of the classical heat equation (with a lower order term). This has made it a popular object of study and the mathematical literature dedicated to the porous medium equation is vast. A good  survey of the field is the monograph \cite{V}. 

In this paper we are interested in the inverse boundary value problem associated to \eqref{eq}. Here this amounts to the following question: is the set of pairs of Dirichlet and Neumann data corresponding to a suitably large class of solutions of \eqref{eq} sufficient for the reconstruction of the coefficients $\epsilon$, $\gamma$, and $\lambda$? Historically, the first problem of this kind was proposed by Calder\'on in \cite{Ca}, for the conductivity equation $\nabla\cdot(\gamma\nabla u)=0$. The question of uniqueness in the original Calder\'on inverse boundary value problem (i.e. does the boundary data uniquely determine the coefficient $\gamma$) was answered affirmatively in \cite{Na} for $n=2$ and in  \cite{SU} for $n\geq3$. Since then, there have been many results covering analogous problems for linear and nonlinear equations. The full list of even the most important of these  is too extensive to include here. We restrict ourselves to mentioning some of the important results known for quasilinear and semilinear elliptic and parabolic equations. For semilinear equations, see \cite{CaKi}, \cite{ChKi}, \cite{FO}, \cite{I1}, \cite{IN}, \cite{IS}, \cite{KiUh}, \cite{KU}, \cite{KU2}, \cite{LLLS1}, \cite{LLLS2}, \cite{S2}. For quasilinear equations, not in divergence form, see \cite{I2}. For quasilinear equations in divergence form see \cite{CF1}, \cite{CF2}, \cite{CFKKU}, \cite{CK}, \cite{CNV}, \cite{EPS}, \cite{HS}, \cite{KN}, \cite{MU}, \cite{Sh}, \cite{S1}, \cite{S3}, \cite{SuU}, (also \cite{C} for quasilinear time-harmonic Maxwell systems).

Results for degenerate equations, such as equation \eqref{eq}, are quite few. Examples would include the following works for the weighted p-Laplace equation, \mbox{$\nabla\cdot(a(x)|\nabla u|^{p-2}\nabla u)=0$}, which is a degenerate elliptic quasilinear PDE: \cite{BKS}, \cite{B}, \cite{BHKS}, \cite{GKS}, \cite{BIK}, \cite{KW}. We note that a uniqueness result without additional constraints, such as monotonicity, has not yet been derived for the weighted p-Laplacian. 

An interesting (and also practically useful) variation of the inverse boundary value problem described above is the case of ``partial data''. The goal remains to prove uniqueness for the various coefficients appearing in the equation, but with the known boundary data supported/restricted to a proper subset of the boundary. In the original Calder\'{o}n problem this was solved in \cite{IUY} for $n=2$. For $n\geq3$, with some restrictions on the geometry of the subsets on which the boundary data is known, a result can be found in \cite{KSU}. We would also like to mention here the work \cite{DKSU} on the linearized Calder\'{o}n problem with partial data. For elliptic semilinear and quasilinear problems partial data results have been obtained in \cite{KKU}, \cite{KU}, \cite{KU2}, \cite{LLLS2}.

Finally, we would like to mention one existing result for the porous medium equation, namely \cite{CGN}, where the equation \eqref{eq}, without the absorption term, is considered. Uniqueness in the inverse boundary value problem is obtained when the boundary data are available for infinite time (i.e. $T=\infty$). In the case of the heat equation this is sufficient to then obtain uniqueness for data available on a finite time interval, since solutions are analytic in time. This is not the case for the porous medium equation (see \cite{V}), so uniqueness with a finite time of observation must be obtained by other means, which is one of the goals of this paper.

In this paper we consider the inverse boundary value problem for the equation \eqref{eq}, with a finite time of observation (i.e. $T<\infty$), in both the partial data and the global data cases.

\subsection{Existence of solutions}

Before stating the main result concering the inverse boundary value problem for equation \eqref{eq}, we first need to outline the sense in which we consider a function $u$ to be a solution of \eqref{eq}. The approach is adapted from \cite{V}.

Let $Q_T:=(0,T)\times\Omega$ and $S_T:=(0,T)\times\partial\Omega$. In order to define a suitable space of Dirichlet boundary data. let 
\begin{equation}
C_{t}(Q_T)=\left\{\varphi\in C^\infty(\overline{Q_T}): \supp\varphi\cap \left[\{T\}\times\dom \right]=\emptyset\right\}
\end{equation} 
and let $H^1_{t}(Q_T)$ be the completion of this space in $H^1(Q_T)$. We will  denote by $H^{\frac{1}{2}}_{t}(S_T)$ the subspace of $H^{\frac{1}{2}}(S_T)$ that consists of traces of $H^1_{t}(Q_T)$ functions.
We also introduce
\begin{equation}
C_{\diamond}(Q_T)=\left\{\varphi\in C^\infty(\overline{Q_T}): \supp\varphi\cap \left[S_T\cup\{T\}\times\dom \right]=\emptyset\right\},
\end{equation}
and its completion $H^1_{\diamond}(Q_T)$ in $H^1(Q_T)$.

Let by $\tau_\dom$ the boundary trace operator for the domain $\dom$. It is bounded between $H^1(\dom)$ and $H^{\frac{1}{2}}(\pr\dom)$. Let $\tau_{Q_T}$ be the boundary trace operator for $Q_T$ to the side boundary $S_T$. It is not hard to check that $\tau_{Q_T}$ is bounded from $L^2((0,T);H^1(\dom))$ to $L^2((0,T);H^{\frac{1}{2}}(\pr\dom))$.

\begin{dfn}\label{def-sol}
 We say that $u\in L^\infty(Q_T)$, $u\geq0$, is a weak solution of \eqref{eq} in $Q_T$ if it
satisfies the following conditions:
\begin{enumerate}
\item $\nabla (u^m)$ exists in the sense of distributions and $\nabla (u^m)\in L^2(Q_T)$;
\item for any test function $\varphi\in C_{\diamond}(Q_T)$ (or, equivalently, any $\varphi\in H^1_{\diamond}(Q_T)$)
\begin{equation} \int_{Q_T}\gamma\nabla\varphi\cdot\nabla (u^m)\dd t\dd x+\int_{Q_T}\lambda \varphi u^q\dd t\dd x-\int_{Q_T}\epsilon\partial_t\varphi\, u\dd t\dd x =\int_{Q_T}\varphi f\dd t\dd x;
\end{equation}
\item $\tau_{Q_T}(u^m)=\phi^m$.
\end{enumerate}
\end{dfn}

We will show in section \ref{forward} that
\begin{thm}\label{thm-forward}
If $\phi\in C(\overline{S_T})$, $\phi^m\in C(\overline{S_T})\cap C^{0,1}(S_T)$, $\phi\geq0$, and $\phi(0,x)=0$ for all $x\in\pr\dom$, and if $f\in L^\infty(Q_T)$, $f\geq0$, then there exists a unique weak solution $u$ of \eqref{eq} in $Q_T$. This solution satisfies the energy estimate
\begin{equation}
||u^m||_{L^2((0,T);H^1(\dom))}\leq C(1+T)^{\frac{1}{2}}\left( ||\phi^{m+q}||_{C^{0,1}(\dom)}+||\phi||_{C(\dom)} +||f||_{L^\infty(Q_T)}\right),
\end{equation}
with a constant $C>0$ that depends on $\dom$, $m$, $q$, and the upper and lower bounds of $\epsilon$, $\gamma$, and $\lambda$.
The solution $u$ satisfies the maximum principle
\begin{equation}
0\leq \sup_{Q_T}u \leq \sup_{S_T} \phi.
\end{equation}
If $\phi_1\leq\phi_2$, $f_1\leq f_2$ are as above and give rise to weak solutions $u_1$ and $u_2$, then
\begin{equation}
u_1\leq u_2.
\end{equation}
\end{thm}

We have been unable to find a proof of this exact result in the available literature on the porous medium equation. For this reason, and also for the convenience of the reader, we provide our own proof. The methods needed are not new. We have followed the argument in \cite{V}, with some additional techniques from \cite{Be}.

 If $\psi$ is as in the above theorem and $u$ is the corresponding weak solution,  the Neumann data $\gamma\pr_\nu u^m|_{S_T}$  can be defined by 
\begin{equation}
\la \gamma\pr_\nu u^m|_{S_T}, \psi|_{S_T}\ra=\int_{Q_T}\left(\gamma \nabla\psi\cdot\nabla(u^m)+\lambda\psi u^q-\epsilon\pr_t\psi u \right)\dd t\dd x.
\end{equation}

\subsection{Main results}

We can  now define the Dirichlet-to-Neumann map
\begin{equation}
\Lambda_{\epsilon,\gamma,\lambda}^{PM} (\phi)=\gamma\pr_\nu u^m|_{S_T}\in \left(H^{\frac{1}{2}}_t(S_T)\right)'.
\end{equation}

Suppose  that we have two sets of coefficients, $\epsilon^{(i)}$, $\gamma^{(i)}$, $\lambda^{(i)}$, and $\epsilon^{(ii)}$, $\gamma^{(ii)}$, $\lambda^{(ii)}$, that are as above. The first main result of our paper is the following.
\begin{thm}\label{thm-inverse}
If $\Lambda_{\epsilon^{(i)},\gamma^{(i)},\lambda^{(i)}}^{PM} (\phi)=\Lambda_{\epsilon^{(ii)},\gamma^{(ii)},\lambda^{(ii)}}^{PM} (\phi)$, for all $\phi\in C(\overline{S_T})$,  such that $\phi^m\in C(\overline{S_T})\cap C^{0,1}_{loc}(S_T)$, $\phi\geq0$, and $\phi(0,x)=0$ for all $x\in\pr\dom$, then $\gamma^{(i)}=\gamma^{(ii)}$, $\epsilon^{(i)}=\epsilon^{(ii)}$, and $\lambda^{(i)}=\lambda^{(ii)}$.
\end{thm}

In the case when the $\gamma$ coefficient is a priori known to be constant, we also have the following partial data result.
\begin{thm}\label{thm-partial}
 Let  $\sig\subset\pr\dom$ be open. If $\Lambda_{\epsilon^{(i)},1,\lambda^{(i)}}^{PM} (\phi)=\Lambda_{\epsilon^{(ii)},1,\lambda^{(ii)}}^{PM} (\phi)$, for all $\phi\in C(\overline{S_\infty})$,  such that $\supp\phi\subset[0,T]\times\sig$, $\phi^m\in C(\overline{S_\infty})\cap C^{0,1}_{loc}(S_\infty)$, $\phi\geq0$, and $\phi(0,x)=0$ for all $x\in\pr\dom$, then  $\epsilon^{(i)}=\epsilon^{(ii)}$ and $\lambda^{(i)}=\lambda^{(ii)}$.
\end{thm}

We give a proof of Theorem \ref{thm-inverse} in section \ref{inverse}. The main idea for the proof is to use two sucessive transformations of equation \eqref{eq}. The first is the change of function $v=u^m$, which gives us the equation
\begin{equation}
\epsilon\pr_t v^{\frac{1}{m}}-\nabla\cdot(\gamma\nabla v)+\lambda v^{\frac{q}{m}}=0.
\end{equation}
The second transformation involves the function
\begin{equation}
V(x)=\int_0^T(T-t)^\alpha v(t,x)\dd t,
\end{equation}
where $\alpha,T>0$ will be chosen in the course of the proof. Since the equation satisfied by $v$ is non-linear, we do not obtain a closed form PDE for $V$. It is however possible to show that $V$ satisfies a differential inequality of the form
\begin{equation}
0\leq\nabla\cdot(\gamma(x)\nabla V(x))\leq C_1\epsilon(x)(V(x))^{\frac{1}{m}}+C_2\lambda(x)(V(x))^{\frac{q}{m}}.
\end{equation}
If we choose Dirichlet boundary data so that $v|_{S_T}(t,x)=ht^m g(x)$, with $h$ a large parameter, we can deduce the first few terms in the asymptotic expansion of $V$ as $h\to\infty$. The Neumann data for each of these terms is determined by the Dirichlet-to-Neumann map, and we can use this information to separately show uniqueness for each of the $\epsilon$, $\gamma$, and $\lambda$ coefficients.

We prove Theorem \ref{thm-partial} in section \ref{partial}. The starting point for the argument is an integral identity derived in section \ref{inverse}. We reduce this integral identity to a problem similar to the linearized local Calder\'{o}n problem considered in \cite{DKSU}, but with an $L^1$ coefficient function rather than the $L^\infty$ one. The result follows by adapting the original argument to this new situation.

Finally, we would like to remark on the importance of the partial data case. Of course, one reason to consider this case is that in possible real life applications only partial data may be practically available. Here however there is an additional theoretical reason, namely that in the partial data case the equation remains in a degenerate regime for the entire time of observation, since the solution must remain zero on a part of the boundary. This shows that our method does not work by moving the equation to a non-degenerate regime, but rather that it can handle the equation even in situations that are not covered by previous works.

\section{The forward problem}\label{forward}

The main idea for the proof of existence of weak solutions to \eqref{eq} comes from the intuition that the maximum principle should hold for such solutions. If the initial data, the Dirichlet boundary data, and the source term are strictly positive (an bounded), and the maximum/minimum principle holds, then the solution would not take values close to zero and infinity. When this is the case, we can modify the equation so as to remove its singularities, while keeping the same solution. This is the same method used in the proof of \cite[Theorem 5.14]{V}.

\subsection{Existence of solutions}

We will now make the above  rigorous. Let $a_k(x,z)$, $b_k(x,z)$, $k=1,2,3,\ldots$,  be smooth, bounded functions, with the $a_k$ also having strictly positive lower bounds, and such that when
\begin{equation}\frac{1}{k}\leq z\leq \sup_{S_T}\phi+T\sup_{Q_T} f+\frac{1+T}{k},
\end{equation}
we have
\begin{equation}
a_k(x,z)=m\gamma(x)z^{m-1}, \quad b_k(x,z)=\lambda(x)z^q.
\end{equation}
We also choose functions $f_k\in C^\infty(Q_T)$ such that
\begin{equation}
f(t,x)\leq f_{k+1}(t,x)\leq f_k(t,x)\leq f(t,x)+\frac{1}{k}.
\end{equation}
Let now $u_k$ be the solutions of 
\begin{equation}\label{eq-k}
\left\{\begin{array}{l}\epsilon\pr_t u_k -\nabla\cdot\left(a_k(x,u_k)\nabla u_k)\right)+b_k(x,u_k)=f_k,\\[5pt] u_k(0,x)=\frac{1}{k},\quad u_k|_{[0,T)\times\pr\dom}=\phi+\frac{1}{k}.\end{array}\right.
\end{equation}
The problem \eqref{eq-k} has a solution $u_k \in C^{1,2}(Q_T)\cap C(\overline{Q_T})$ (see \cite{La}, or \cite{Li}), which furthermore satisfies the maximum principle
\begin{equation}\label{mp-k}
\frac{1}{k}\leq u_k(t,x)\leq\frac{1}{k}+\sup_{S_T}\phi+t\sup_{Q_T}f_k.
\end{equation}
It follows that $u_k$ is also a solution to 
\begin{equation}\label{eq-k-2}
\left\{\begin{array}{l}\epsilon\pr_t u_k -\nabla\cdot(\gamma\nabla u^m_k)+\lambda u_k^q=f_k,\\[5pt] u_k(0,x)=\frac{1}{k},\quad u_k|_{[0,T)\times\pr\dom}=\phi+\frac{1}{k}.\end{array}\right.
\end{equation}

By \eqref{mp-k}, $u_{k}$ is also a solution to $\epsilon\pr_t u_{k} -\nabla\cdot\left(a_{k+1}(x,u_k,\nabla u_k)\right)=f_k$. By the comparison principle (see \cite[Theorem 9.7]{Li}), we have that
\begin{equation}
0\leq u_{k+1}\leq u_{k},\quad\forall k=1,2\ldots.
\end{equation}
Let $u$ be the pointwise limit
\begin{equation}
u(t,x)=\lim_{k\to\infty} u_k(t,x).
\end{equation}
It is clear that $u\in L^\infty(\dom)$ and
\begin{equation}
0\leq u(t,x)\leq \sup_{S_T}\phi+t\sup_{Q_T} f.
\end{equation}
A simple application of the monotone convergence theorem gives that in $L^2(Q_T)$ norm we have $u_k\to u$, $u_k^m\to u^m$, and $u_k^q\to u^q$.

Let $\tilde\phi:Q_T\to [0,\infty)$ be a smooth extension of the Dirichlet data, such that we still have $\tilde\phi(0,x)=0$ for all $x\in\dom$ and
\begin{equation}
||\tilde\phi||_{W^{1,\infty}(Q_T)}\leq C||\phi||_{C^{0,1}(S_T)},
\end{equation}
with a constant $C>0$.
Let 
\begin{equation}
\eta_k=u_k^m-\left(\tilde\phi+\frac{1}{k}\right)^m,
\end{equation}
which is zero on $S_T$. We will multiply \eqref{eq-k-2} by $\eta_k$ and integrate over $Q_T$. Note that
\begin{multline}
-\int_{Q_T}\eta_k\nabla(\gamma\nabla u_k^m)\dd t\dd x\\[5pt]
=\int_{Q_T}\gamma|\nabla u_k^m|^2\dd t\dd x
-\int_{Q_T}\gamma\nabla \left(\tilde\phi+\frac{1}{k}\right)^m\cdot\nabla u_k^m\dd t\dd x,
\end{multline}
and
\begin{multline}
\int_{Q_T}\epsilon\pr_t u_k\eta_k\dd t\dd x
=\int_{\dom}\frac{\epsilon}{m+1}u_k^{m+1}(T)\dd x\\[5pt]
-\int_{\dom}\epsilon u_k(T)\left(\tilde\phi(T)+\frac{1}{k}\right)^m\dd x
+\int_{Q_T}\epsilon u_k\pr_t\left(\tilde\phi+\frac{1}{k}\right)^m\dd t\dd x.
\end{multline}
It follows that
\begin{multline}
\int_{Q_T}\gamma|\nabla u_k^m|^2\dd t\dd x +\int_{Q_T}\lambda u_k^{m+q}\dd t\dd x
+\int_{\dom}\frac{\epsilon}{m+1}u_k^{m+1}(T)\dd x\\[5pt]
=\int_{Q_T}\gamma\nabla \left(\tilde\phi+\frac{1}{k}\right)^m\cdot\nabla u_k^m\dd t\dd x
+\int_{\dom}\epsilon u_k(T)\left(\tilde\phi(T)+\frac{1}{k}\right)^m\dd x\\[5pt]
-\int_{Q_T}\epsilon u_k\pr_t\left(\tilde\phi+\frac{1}{k}\right)^m\dd t\dd x
+\int_{Q_T}f_k\left[u_k^m-\left(\tilde\phi+\frac{1}{k}\right)^m\right]\\[5pt]
+\int_{Q_T}\lambda u_k^q\left(\tilde\phi+\frac{1}{k}\right)^m\dd t\dd x.
\end{multline}
After straightforward estimates we obtain that
\begin{multline}
\int_{Q_T}|\nabla u_k^m|^2\dd t\dd x<C'(1+T)\Bigg(\left\Vert\left(\phi+\frac{1}{k}\right)^m\right\Vert_{C^{0,1}(S_T)}^{2}
+\left\Vert\left(\phi+\frac{1}{k}\right)^m\right\Vert_{C(S_T)}\\[5pt] +\left\Vert\left(\phi+\frac{1}{k}\right)^{m+q}\right\Vert_{C(S_T)}+\left\Vert\phi+\frac{1}{k}\right\Vert_{C(S_T)}^2 +\left\Vert f+\frac{1}{k}\right\Vert_{L^\infty(Q_T)}^2 \Bigg),
\end{multline}
where $C'>0$ is a constant independent of $k$ and $T$.

It follows that (a subsequence of) $\nabla u_k^m$ converges weakly in $L^2(Q_T)$ to a limit $U$. It is easy to see that $U=\nabla u^m$ in the sense of distributions.
Since $\tau_{Q_T}(u_k^m)(t)\to \phi^m$ in $L^2((0,T); H^{\frac{1}{2}}(\pr\dom))$ by construction, and at the same time $\tau_{Q_T}(u_k^m)\rightharpoonup \tau_{Q_T}(u^m)$, it follows that
\begin{equation}
\tau_{Q_T}(u^m)=\phi^m.
\end{equation}
Therefore $u$ is a weak solution of \eqref{eq} in $Q_T$.

\subsection{Uniqueness of solutions}

Before proving the energy estimate and the comparison principle part of Theorem \ref{thm-forward}, we first need to prove the uniqueness of weak solutions. Once uniqueness has been established, we can deduce the desired properties from the corresponding ones that hold for $u_k$, since we would then know that each weak solution can be obtained as a limit of such approximate soultions.
The following argument uses ideas from  the proofs of \cite[Theorem 6.5, Theorem 6.6]{V} and from \cite[Section 3]{Be}.

Suppose $u_1$, $u_2$ are both weak solutions of \eqref{eq} in $Q_T$, with possibly different boundary data $\phi_1\leq\phi_2$ and source terms $f_1$ and $f_2$. Then for any $\varphi\in C_\diamond(Q_T)\cap C^\infty(Q_T)$ we have
\begin{multline}\label{eq-2-sol}
-\int_{Q_T}\epsilon \pr_t\varphi(u_1-u_2)\dd t\dd x+\int_{Q_T}\lambda\varphi(u_1^q-u_2^q)\dd t\dd x\\[5pt]
\leq\int_{Q_T} \nabla\cdot(\gamma\nabla\varphi) (u_1^m-u_2^m)\dd t\dd x
+\int_{Q_T}\varphi (f_1-f_2)\dd t\dd x.
\end{multline}
Let 
\begin{equation}
a(t,x)=\left\{\begin{array}{l}\frac{u_1^m(t,x)-u_2^m(t,x)}{u_1-u_2},\quad\text{if }u_1(t,x)\neq u_2(t,x),\\[5pt] 0,\quad\text{if }u_1(t,x)=u_2(t,x).\end{array}\right.
\end{equation}
The function $a$ is continuous, but may not be smooth. For $k=(k_1,k_2)\in\mathbb{N}^2$, let $a_k\in C^\infty(Q_T)$ be such that 
\begin{equation}
\frac{1}{k_1}\leq a_k\leq k_2,
\end{equation}
and $a_k\to a$.

 Let $\theta\in C^\infty_0(Q_T)$, $\theta\geq 0$ be an arbitrary function and $\varphi_k$ be the unique smooth solution of the backwards in time linear parabolic problem in $Q_T$
\begin{equation}\label{eq-phi}
\left\{
\begin{array}{l} \epsilon\pr_t\varphi_k+a_k\nabla\cdot(\gamma\nabla\varphi_k)+\theta=0,\\[5pt]
\varphi_k(T,x)=0,\;\forall x\in\dom,\quad \varphi_k|_{S_T}=0.
\end{array}
\right.
\end{equation}
By the maximum principle we have that $\varphi_k\geq0$.
Using $\varphi_k$ as a test function in \eqref{eq-2-sol} we get
\begin{multline}\label{eq-pair}
\int_{Q_T}\theta (u_1-u_2)\dd t\dd x+\int_{Q_T}\lambda\varphi_k(u_1^q-u_2^q)\dd t\dd x\\[5pt]
\leq\int_{Q_T}(a-a_k)\nabla\cdot(\gamma\nabla\varphi_k) (u_1-u_2)\dd t\dd x
+\int_{Q_T}\varphi_k (f_1-f_2)\dd t\dd x.
\end{multline}

Let 
\begin{equation}
J_k=\int_{Q_T}|u_1-u_2|\,|a-a_k|\,|\nabla\cdot(\gamma\nabla\varphi_k)|\dd t\dd x
\end{equation}
and note immediately that
\begin{equation}\label{J-estimate}
J_k\leq\left( \int_{Q_T}a_k[\nabla\cdot(\gamma\nabla\varphi_k)]^2\dd t\dd x  \right)^{\frac{1}{2}}\left(  \int_{Q_T}|u_1-u_2|^2\frac{|a-a_k|^2}{a_k}\dd t\dd x \right)^{\frac{1}{2}}.
\end{equation}
We will control each of the two factors on the right hand side separately.

Let $\zeta:[0,T]\to[\frac{1}{2},1]$ be a smooth function such that $\zeta'(t)\geq c>0$. We multiply \eqref{eq-phi} by $\frac{\zeta}{\epsilon} \nabla\cdot(\gamma\nabla\varphi_k)$ and integrate to obtain
\begin{equation}
\int_{Q_T}\pr_t\varphi_k\zeta\nabla\cdot(\gamma\nabla\varphi_k)\dd t\dd x
+\int_{Q_T}\frac{\zeta}{\epsilon} a_k[\nabla\cdot(\gamma\nabla\varphi_k)]^2\dd t \dd x
+\int_{Q_T}\zeta\frac{\theta}{\epsilon}\nabla\cdot(\gamma\nabla\varphi_k)\dd t\dd x =0.
\end{equation}
We have that
\begin{multline}
\int_{Q_T}\pr_t\varphi_k\zeta\nabla\cdot(\gamma\nabla\varphi_k)\dd t\dd x
=-\int_{Q_T} \zeta\gamma\nabla\varphi_k\cdot\nabla(\pr_t\varphi_k)\dd t\dd x\\[5pt]
=-\frac{1}{2}\int_{Q_T}\zeta\gamma\pr_t(\nabla\varphi_k)^2\dd t\dd x
\geq \frac{1}{2}\int_{Q_T}\zeta'\gamma (\nabla\varphi_k)^2\dd t\dd x.
\end{multline}
Then
\begin{equation}
\frac{1}{2}\int_{Q_T}\zeta'\gamma |\nabla\varphi_k|^2\dd t\dd x+
\int_{Q_T}\frac{\zeta}{\epsilon} a_k[\nabla\cdot(\gamma\nabla\varphi_k)]^2\dd t \dd x
\leq \int_{Q_T}\zeta\gamma\nabla\frac{\theta}{\epsilon}\cdot\nabla\varphi_k\dd t\dd x.
\end{equation}
Estimating
\begin{equation}
\int_{Q_T}\zeta\gamma\nabla\frac{\theta}{\epsilon}\cdot\nabla\varphi_k\dd t\dd x
\leq \frac{c}{4}\int_{Q_T}\gamma |\nabla\varphi_k|^2\dd t\dd x+\frac{1}{c}\int_{Q_T}\gamma |\nabla\frac{\theta}{\epsilon}|^2\dd t\dd x,
\end{equation}
we conclude that
\begin{equation}\label{phi-bound}
\int_{Q_T}\gamma |\nabla\varphi_k|^2\dd t\dd x+
\int_{Q_T} a_k[\nabla\cdot(\gamma\nabla\varphi_k)]^2\dd t \dd x
\leq C \int_{Q_T}\gamma |\nabla\frac{\theta}{\epsilon}|^2\dd t\dd x,
\end{equation}
with a constant $C>0$ which do not depend on $k$.

The estimate for  the second factor on the right hand side of \eqref{J-estimate} relies only on the strategy for approximating $a$ by $a_k$. There is no difference at all between our case here and the argument in \cite{V}, so we quote the result
\begin{equation}
\int_{Q_T}|u_1-u_2|^2\frac{|a-a_k|^2}{a_k}\dd t\dd x\leq \frac{C}{k_1}.
\end{equation}
It follows that in the limit $J_k\to0$.

Let $t_0\in(0,T)$, $0\leq\psi(x)\leq1$ be a smooth function, and $\rho_\eta$, $\eta>0$, be a standard mollifier. We set
\begin{equation}
\theta(t,x)=\epsilon(x)\psi(x)\rho_\eta(t-t_0).
\end{equation} 
Then 
\begin{equation}
\varphi(t,x)=\int_t^T\rho_\eta(s-t_0)\dd s
\end{equation}
is a supersolution for \eqref{eq-phi}, so for all $k$ we have that $\varphi_k\leq\varphi$.

Taking the limit in \eqref{eq-pair}, we obtain 
\begin{multline}
\int_{Q_T} \epsilon\psi\rho_\eta(u_1-u_2)\dd t\dd x\leq  \int_{Q_T}\lambda(u_2^q-u_1^q)_+\left(\int_t^T\rho_\eta(s-t_0)\dd s\right)\dd t\dd x\\[5pt]
+\int_{Q_T} (f_1-f_2)_+\left(\int_t^T\rho_\eta(s-t_0)\dd s\right)\dd t\dd x.
\end{multline}
Taking the $\eta\to0$ limit, we get
\begin{equation}
\int_\dom \epsilon\psi (u_1(t_0)-u_2(t_0))_+\dd x\leq \int_0^{t_0}\int_{\dom}\lambda(u_2^q-u_1^q)_+\dd t\dd x
+\int_0^{t_0}\int_{\dom}(f_1-f_2)_+\dd t\dd x,
\end{equation}
which easily gives that
\begin{equation}\label{L1-estimate}
\int_\dom  (u_1(t_0)-u_2(t_0))_+\dd x\leq C\left(\int_0^{t_0}\int_{\dom}(u_2^q-u_1^q)_+\dd t\dd x
+\int_0^{t_0}\int_{\dom}(f_1-f_2)_+\dd t\dd x\right).
\end{equation}

We need to distinguish two cases. The first is when $q\geq1$. In this case we can from the start assume that $\phi_1=\phi_2$ and $f_1=f_2$. Note that both $u_1$ and $u_2$ are bounded, so there exists $M>0$ such that
\begin{equation}
(u_2^q-u_1^q)_+\leq M (u_2-u_1)_+.
\end{equation}
Then
\begin{multline}
\int_\dom  (u_1(t_0)-u_2(t_0))_+\dd x\leq C\int_0^{t_0}\int_{\dom}(u_2(t)-u_1(t))_+\dd t\dd x\\[5pt]
\leq C\int_0^{t_0}\int_0^t\int_\dom(u_1(s)-u_2(s))_+\dd s\dd t\dd x\\[5pt]
\leq Ct_0  \int_0^{t_0}\int_{\dom}(u_1(s)-u_2(s))_+\dd s\dd x.
\end{multline}
Gronwall's inequality implies that
\begin{equation}
\int_\dom  (u_1(t_0)-u_2(t_0))_+\dd x=0,\quad\forall t_0\in(0,T),
\end{equation}
and by exchanging the roles of $u_1$ and $u_2$ we conclude that $u_1=u_2$ everywhere.


The second case is when $0<q<1$.  Here we would like to take $u_1=\hat u$ to be any weak solution of \eqref{eq} in $Q_T$, with boundary data $\phi$ and source term $f$, and $u_2$ to be the approximate solution $u_k$ constructed above in the existence step. There is a slight difference to the assumtions above since $u_k$ has positive initial data $\frac{1}{k}$. Going over the argument  with this in mind quickly shows that an additional term of the form
\begin{equation}
\int_\dom (-u_k(0,x))_+\dd x
\end{equation}
should appear on the right hand side of \eqref{L1-estimate}. Since this term is actually zero, we may continue without modifying anything. 

There is another modification we wish to make. Let $\omega>0$ and let $u_1= e^{\omega t}\hat u$, $u_2= e^{\omega t}u_k$. It is easy to see that $u_1$ satisfies the equation
\begin{equation}
\pr_t u_1-\nabla\cdot(\gamma\nabla u_1)+\lambda u_1^q-\omega u_1=e^{\omega t}f
\end{equation}
in the weak sense, with a similar situation holding for $u_2$. 
If we apply \eqref{L1-estimate} to these choices of $u_1$ and $u_2$ we  have
\begin{multline}
e^{\omega t_0}\int_\dom (\hat u(t_0)-u_k(t_0))_+\leq
C\left(\int_0^{t_0}\int_\dom e^{\omega t}(u_k^q-\hat u^q-\omega(u_k-\hat u))_+\dd t\dd x\right.\\[5pt]\left.
+\int_0^{t_0}\int_\dom e^{\omega t}(f-f_k)_+\dd t\dd x\right).
\end{multline}
The last term is zero. If
\begin{equation}\label{plus-condition}
u_k^q(t,x)-\hat u^q(t,x)-\omega(u_k(t,x)-\hat u(t,x))\geq0,
\end{equation}
and $u_k(t,x)\geq \hat u(t,x)$, then we must have that 
\begin{equation}
\omega\leq\frac{u_k^q(t,x)-\hat u^q(t,x)}{u_k(t,x)-\hat u(t,x)}\leq(u_k(t,x)-\hat u(t,x))^{q-1}
\leq k^{1-q}.
\end{equation}
If we choose $\omega>k^{1-q}$, then we see that \eqref{plus-condition} can only hold when $\hat u(t,x)\geq u_k(t,x)$. In this case we have
\begin{equation}
e^{\omega t_0}\int_\dom (\hat u(t_0)-u_k(t_0))_+\leq C\omega\int_0^{t_0}\int_\dom e^{\omega t}(\hat u-u_k)_+\dd t\dd x.
\end{equation}
By Gronwall's inequality we conclude that $\hat u\leq u_k$, and therefore $\hat u\leq u$, where $u$ is the solution constructed in the previous subsection.

Returning to \eqref{L1-estimate} we get that
\begin{equation}
\int_\dom (u(t_0)-\hat u(t_0))\dd x\leq C\int_0^{t_0}\int_\dom (\hat u^q- u^q)_+\dd t\dd x=0,
\end{equation}
so $u=\hat u$.

\subsection{Energy inequality and maximum principle}

Once uniqueness of solutions has been established, we can derive properties of the weak solutions of \eqref{eq} from properties of the approximate solutions $u_k$.

The energy inequality follows easily once we note that 
\begin{multline}
||\nabla u^m||_{L^2(Q_T)}\leq \liminf_{k\to\infty} ||\nabla u_k^m||_{L^2(Q_T)}\\[5pt]
\leq C(1+T)^{\frac{1}{2}}\left(||\phi^{m+q}||_{C^{0,1}(S_T)}+||\phi||_{C(S_T)} + ||f||_{L^\infty(Q_T)}\right).
\end{multline}

Regarding the comparison principle, suppose we have boundary data and sources $\phi_1\leq\phi_2$, $f_1\leq f_2$, with corresponding solutions $u_1$ and $u_2$. The approximate solutions, by the comparison principle for non-degenerate quasilinear elliptic equations (see \cite[Theorem 9.7]{Li})  must satisfy
\begin{equation}
u_{1,k}\leq u_{2,k},
\end{equation}
and this property is preserved in the limit.
This concludes the proof of Theorem \ref{thm-forward}. 

\section{The inverse problem with full data}\label{inverse}

In this section we give a proof of Theorem \ref{thm-inverse}.
We first reformulate the problem by performing the change of function $v(t,x)=u^m(t,x)$. The new function $v$ satisfies
\begin{equation}\label{eq-v}
\left\{\begin{array}{l}\epsilon(x)\pr_t v(t,x)^{\frac{1}{m}} -\nabla\cdot(\gamma(x)\nabla v(t,x))+\lambda v^{\frac{q}{m}}(t,x)=0,\\[5pt] v(0,x)=0,\quad v|_{S_T}=f(t,x),\quad v\geq0,\end{array}\right.
\end{equation}
where $f=\phi^m$.

The notion of weak solutions for \eqref{eq} we have introduced above naturally transforms into a notion of weak solutions for \eqref{eq-v}, in the space-time domain $Q_T$. The Dirichlet-to-Neumann map $\Lambda^{PM}_{\epsilon,\gamma,\lambda}$ uniquely determines the Dirichlet-to-Neumann map
\begin{equation}
\Lambda^v_{\epsilon,\gamma,\lambda}(f)=\gamma\pr_\nu v|_{S_T},
\end{equation}
associated to the equation \eqref{eq-v}. Here $f\in C(\overline{S_T})\cap C^{0,1}(S_T)$, $f\geq0$, and $f(0,x)=0$ for all $x\in\pr\dom$.

\subsection{Time-integral transform and basic estimates}
Let $T,\alpha>0$, $h>1$. We choose the boundary data to be of the form 
\begin{equation}
f(t,x)=t^m\,h\,g(x),\quad g\geq0.
\end{equation}
Later we will fix $\alpha$ and $T$ to particular values, so we will not emphasize in the following the dependence of various quantities on these. Let
\begin{equation}
V(x)=\int_0^T(T-t)^\alpha v(t,x)\dd t.
\end{equation}
Then
\begin{equation}
\nabla\cdot(\gamma\nabla V(x))=\N_t(x)+\N_a(x),
\end{equation}
where
\begin{equation}
\N_t(x)=\epsilon(x)\alpha\int_0^T(T-t)^{\alpha-1} v^{\frac{1}{m}}(t,x)\dd t,
\end{equation}
and
\begin{equation}
\N_a(x)=\lambda(x)\int_0^T (T-t)^\alpha v^{\frac{q}{m}}(t,x)\dd t.
\end{equation}
Since
\begin{equation}
\int_0^T(T-t)^\alpha t^m\dd t=T^{1+\alpha+m}\frac{\Gamma(1+\alpha)\Gamma(1+m)}{\Gamma(2+\alpha+m)},
\end{equation}
we have that
\begin{equation}
\left\{\begin{array}{l}\nabla\cdot(\gamma\nabla V)=\N_t+\N_a,\\[5pt] V|_{\pr\dom}=hT^{1+\alpha+m}\frac{\Gamma(1+\alpha)\Gamma(1+m)}{\Gamma(2+\alpha+m)}g.\end{array}\right.
\end{equation}
In connection to this equation we introduce the Dirichlet-to-Neumann map
\begin{equation}
\Lambda^h_{\epsilon,\gamma,\lambda}(g)=h^{-1}\gamma\pr_\nu V|_{\pr\dom},
\end{equation}
which is determined by $\Lambda^v_{\epsilon,\gamma,\lambda}$ and hence also by $\Lambda^{PM}_{\epsilon,\gamma,\lambda}$.

By H\"older's inequality we have
\begin{multline}
0\leq\N_t(x)=\epsilon(x)\alpha\int_0^T (T-t)^{\frac{\alpha-m'}{m'}}\left[(T-t)^\alpha v\right]^{\frac{1}{m}}\dd t\\[5pt]
\leq\epsilon(x)\alpha\left(\int_0^T(T-t)^{\alpha-m'}\dd t  \right)^{\frac{1}{m'}}\left( V(x)\right)^{\frac{1}{m}}\\[5pt]
=T^{\frac{\alpha}{m'}-\frac{1}{m}}\frac{\alpha}{(\alpha-m'+1)^{\frac{1}{m'}}}\epsilon(x)\left( V(x)\right)^{\frac{1}{m}}.
\end{multline}
Similarly
\begin{equation}
0\leq\N_a(x)\leq T^{\alpha+\frac{m}{m-q}}\left(\alpha\left(1-\frac{q}{m}\right)+1   \right)^{-\frac{m}{m-q}}\lambda(x)(V(x))^{\frac{q}{m}}.
\end{equation}
Note that the above can only work if $\alpha>m'-1=\frac{1}{m-1}$, which we will assume from now on.

Then (choose $p>n$) we get
\begin{equation}
||\N_t||_{L^\infty(\dom)}\leq CT^{\frac{\alpha}{m'}-\frac{1}{m}}|| V||_{W^{1,p}(\dom)}^{\frac{1}{m}}
\end{equation}
and
\begin{equation}
||\N_a||_{L^\infty(\dom)}\leq CT^{\alpha+\frac{m}{m-q}}|| V||_{W^{1,p}(\dom)}^{\frac{q}{m}}.
\end{equation}
By elliptic estimates we have
\begin{multline}
||V||_{W^{2,p}(\dom)}\leq C\bigg( hT^{1+\alpha+m}||g||_{W^{2-\frac{1}{p},p}(\dom)}+T^{\frac{\alpha}{m'}-\frac{1}{m}}|| V||_{W^{1,p}(\dom)}^{\frac{1}{m}}\\[5pt]+T^{\alpha+\frac{m}{m-q}}|| V||_{W^{1,p}(\dom)}^{\frac{q}{m}}\bigg),
\end{multline}
so, since $h>1$, 
\begin{multline}
\max (||V||_{W^{2,p}(\dom)},h)\\[5pt]\leq
C\bigg(hT^{1+\alpha+m}||g||_{W^{2-\frac{1}{p},p}(\dom)}+h+T^{\frac{\alpha}{m'}-\frac{1}{m}}\left[\max(|| V||_{W^{1,p}(\dom)},h)\right]^{\frac{1}{m}}\\[5pt] +T^{\alpha+\frac{m}{m-q}}\left[\max(|| V||_{W^{1,p}(\dom)},h)\right]^{\frac{q}{m}}   \bigg)\\[5pt]
 \leq C\bigg(hT^{1+\alpha+m}||g||_{W^{2-\frac{1}{p},p}(\dom)}+h+(T^{\frac{\alpha}{m'}-\frac{1}{m}}+T^{\alpha+\frac{m}{m-q}})\left[\max(|| V||_{W^{1,p}(\dom)},h)\right]    \bigg).
\end{multline}
Since $\frac{\alpha}{m'}-\frac{1}{m}, \alpha+\frac{m}{m-q}>0$, we will choose $T$ small enough to be able to absorb the $T^{\frac{\alpha}{m'}-\frac{1}{m}}+T^{\alpha+\frac{m}{m-q}}$ term into the left hand side. Then, supressing explicit dependence on $T$, we have
\begin{equation}
||V||_{W^{2,p}(\dom)}\leq hC\left( ||g||_{W^{2-\frac{1}{p},p}(\dom)} +1\right).
\end{equation}

\subsection{Asymptotic expansion to first order in $h$}

We make the Ansatz
\begin{equation}
V(x)=h \frac{T^{1+\alpha+m}\Gamma(1+\alpha)\Gamma(1+m)}{\Gamma(2+\alpha+m)}V_0(x)+R_1(x),
\end{equation}
where
\begin{equation}
\left\{\begin{array}{l}\nabla\cdot(\gamma\nabla V_0)=0,\\[5pt] V_0|_{\pr\dom}=g.\end{array}\right.
\end{equation}
Then $R_1$ must satisfy
\begin{equation}
\left\{\begin{array}{l}\nabla\cdot(\gamma\nabla R_1)=\N_t+\N_a,\\[5pt] R_1|_{\pr\dom}=0.\end{array}\right.
\end{equation}
Let $\sigma=\frac{1}{m}\max(1,q)<1$. We have the estimate
\begin{equation}
||R_1||_{W^{2,p}(\dom)}\leq C h^{\sigma}\left( ||g||_{W^{2-\frac{1}{p},p}(\dom)} +1\right)^{\sigma}.
\end{equation}

The consequence for the Dirichlet-to-Neumann map is that as $h\to\infty$
\begin{equation}
\Lambda^h_{\epsilon,\gamma,\lambda}(g)=\frac{T^{1+\alpha+m}\Gamma(1+\alpha)\Gamma(1+m)}{\Gamma(2+\alpha+m)}\gamma\pr_\nu V_0|_{\pr\dom}+\Ord(h^{1-\sigma}).
\end{equation}
It follows that the DN map 
\begin{equation}
\Lambda_\gamma(g)=\gamma\pr_\nu V_0|_{\pr\dom},
\end{equation}
which is the DN map for the original Calder\'on problem, is determined by $\Lambda^h_{\epsilon,\gamma}$. By \cite{Na} in the $n=2$ case or \cite{SU} in the $n\geq3$ case,  we now have uniqueness for $\gamma$.

Before moving on, note that by the maximum/minimum principle we have $V_0\geq0$ and $R_1\leq 0$. For our convenience below we introduce here
\begin{equation}
v_0(t,x)=ht^mV_0(x),
\end{equation}
which is such that
\begin{equation}
\int_0^T (T-t)^\alpha v_0(t,x)\dd t=h\frac{T^{1+\alpha+m}\Gamma(1+\alpha)\Gamma(1+m)}{\Gamma(2+\alpha+m)} V_0(x).
\end{equation}

\subsection{Asymptotic expansion to second order in $h$}

Here we refine the Ansatz for $V$ to
\begin{equation}
V(x)=h V_0(x)+h^{\frac{1}{m}}V_t(x)+h^{\frac{q}{m}}V_a(x)+R_2(x),
\end{equation}
where $V_0$ is as above and
\begin{equation}
\left\{\begin{array}{l}\nabla\cdot(\gamma\nabla V_t)=h^{-\frac{1}{m}}\N_{0t},\\[5pt] V_t|_{\pr\dom}=0,\end{array}\right.
\end{equation}
\begin{equation}
\left\{\begin{array}{l}\nabla\cdot(\gamma\nabla V_a)=h^{-\frac{1}{m}}\N_{0a},\\[5pt] V_a|_{\pr\dom}=0,\end{array}\right.
\end{equation}
with
\begin{equation}
\N_{0t}(x)=\alpha\epsilon(x)\int_0^T(T-t)^{\alpha-1}v_0^{\frac{1}{m}}(t,x)\dd t
=h^{\frac{1}{m}}\epsilon(x)\frac{T^{\alpha+1}}{\alpha+1}V_0^{\frac{1}{m}}(x)
\end{equation}
and
\begin{equation}
\N_{0a}(x)=\lambda(x)\int_0^T(T-t)^{\alpha}v_0^{\frac{q}{m}}(t,x)\dd t
=h^{\frac{q}{m}}\lambda(x)\frac{T^{\alpha+1+q}}{\alpha+1+q}V_0^{\frac{q}{m}}(x).
\end{equation}
Clearly, $V_t$ and $V_a$  are independent of $h$ and, by the maximum principle, $V_t, V_a\leq0$.

The remainder term $R_2$ must satisfy
\begin{equation}
\left\{\begin{array}{l}\nabla\cdot(\gamma\nabla R_2)=(\N_t-\N_{0t})+(\N_a-\N_{0a}),\\[5pt] R_2|_{\pr\dom}=0.\end{array}\right.
\end{equation}
Note that
\begin{equation}
\epsilon\pr_tv_0^\frac{1}{m}-\nabla\cdot(\gamma\nabla v_0)=h^{\frac{1}{m}}\epsilon V_0^{\frac{1}{m}}(x)\geq0,\quad v_0|_{S_T}=v|_{S_T},
\end{equation}
so $v_0$ is a supersolution and therefore we have that $v_0\geq v$. It follows then that
\begin{equation}
\N_{0t}\geq \N_t,\quad \N_{0a}\geq \N_a.
\end{equation}

Using the same H\"older inequality trick applied above
\begin{multline}
0\leq \N_{0t}-\N_t=\alpha\epsilon\int_0^T(T-t)^{\alpha-1}(v_0^{\frac{1}{m}}-v^{\frac{1}{m}})\dd t\\[5pt]
\leq \alpha\epsilon\int_0^T(T-t)^{\alpha-1} (v_0-v)^{\frac{1}{m}}\dd t\\[5pt]
\leq  T^{\frac{\alpha}{m'}-\frac{1}{m}}\frac{\alpha}{(\alpha-m'+1)^{\frac{1}{m'}}}\epsilon(x) \left(h\frac{T^{1+\alpha+m}\Gamma(1+\alpha)\Gamma(1+m)}{\Gamma(2+\alpha+m)} V_0(x)-V(h,x)\right)^{\frac{1}{m}}\\[5pt]
=T^{\frac{\alpha}{m'}-\frac{1}{m}}\frac{\alpha}{(\alpha-m'+1)^{\frac{1}{m'}}}\epsilon(x)(-R_1(x))^{\frac{1}{m}}.
\end{multline}
Similarly
\begin{equation}
0\leq \N_{0a}-\N_a\leq T^{\alpha+\frac{m}{m-q}}\left(\alpha\left(1-\frac{q}{m}\right)+1   \right)^{-\frac{m}{m-q}}\lambda(x)(-R_1(x))^{\frac{q}{m}}
\end{equation}
It follows that
\begin{equation}
||(\N_t-\N_{0t})+(\N_a-\N_{0a})||_{L^\infty(\dom)}=\Ord(h^{\sigma^2}),
\end{equation}
so, by elliptic estimates we have
\begin{equation}
||R_2||_{W^{2,p}(\dom)}=\Ord(h^{\sigma^2}).
\end{equation}
The Dirichlet-to-Neumann map then has the following expansion as $h\to\infty$
\begin{multline}\label{lambda-expansion}
\Lambda^h_{\epsilon,\gamma,\lambda}(g)=\frac{T^{1+\alpha+m}\Gamma(1+\alpha)\Gamma(1+m)}{\Gamma(2+\alpha+m)}\gamma\pr_\nu V_0|_{\pr\dom}\\[5pt]
+h^{\frac{1}{m}-1}\gamma\pr_\nu V_t|_{\pr\dom}
+h^{\frac{q}{m}-1}\gamma\pr_\nu V_a|_{\pr\dom}
+\Ord(h^{\sigma^2-1}).
\end{multline}

Assume for now that $q\neq 1$. From our assumptions of the range of allowed values for $q$, we have that
\begin{equation}
\frac{1}{m}-1,\frac{q}{m}-1>\sigma^2-1.
\end{equation}
In this case it follows that the Neumann data $\gamma\pr_\nu V_t|_{\pr\dom}$, $\gamma\pr_\nu V_a|_{\pr\dom}$ are each determined individually by $\Lambda^{PM}_{\epsilon,\gamma,\lambda}$.

Let $W$ be any smooth solution to $\nabla\cdot(\gamma\nabla W)=0$. We have that
\begin{multline}
\la \gamma\pr_\nu V_t|_{\pr\dom},W|_{\pr\dom}\ra
=\int_\dom\gamma\nabla V_1\cdot\nabla W\dd x+
\frac{T^{\alpha+1}}{\alpha+1}\int_\dom\epsilon V_0^{\frac{1}{m}}W\dd x\\[5pt]
=\frac{T^{\alpha+1}}{\alpha+1}\int_\dom\epsilon V_0^{\frac{1}{m}}W\dd x.
\end{multline}
We can take $V_0$ to be of the form
\begin{equation}
V_0(x)=1+sH(x),\quad \nabla\cdot(\gamma\nabla H)=0,\;H>0.
\end{equation}
Then
\begin{equation}
m\left.\frac{\dd}{\dd s}\right|_{s=0}\int_\dom\epsilon \left(1+sH(x)\right)^{\frac{1}{m}}W\dd x
=\int_\dom \epsilon HW\dd x
\end{equation}
is determined by $\Lambda^{PM}_{\epsilon,\gamma,\lambda}$. From here it is easy to see that we must have
\begin{equation}
\int_\dom \left(\epsilon^{(i)}-\epsilon^{(ii)} \right)UW\dd x=0
\end{equation}
for all $U,W$ such that $\nabla\cdot(\gamma\nabla U)=\nabla\cdot(\gamma\nabla W)=0$. As shown in \cite{Bu} in the $n=2$ case and in \cite{SU} in the $n\geq3$ case,  this implies that $\epsilon^{(i)}=\epsilon^{(ii)}$. 

Similarly
\begin{equation}
\la \gamma\pr_\nu V_t|_{\pr\dom},W|_{\pr\dom}\ra=\frac{T^{\alpha+1+q}}{\alpha+1+q}\int_\dom\lambda V_0^{\frac{q}{m}}W\dd x
\end{equation}
and an identical argument gives that $\lambda^{(i)}=\lambda^{(ii)}$.

It remains to consider the case $q=1$. In this situation, we have that $\Lambda^{PM}_{\epsilon,\gamma,\lambda}$ determines all quantities of the form
\begin{multline}
\la \gamma\pr_\nu (V_t+V_a)|_{\pr\dom},W|_{\pr\dom}\ra
=\int_\dom V_0^{\frac{1}{m}} W\left(\frac{T^{\alpha+1}}{\alpha+1}\epsilon+ \frac{T^{\alpha+2}}{\alpha+2}\lambda  \right)\dd x.
\end{multline}
We can use our ability to continuously shrink the observation time $T$ in order to determine the $\epsilon$ and $\lambda$ terms individually. This concludes the proof.

\section{The inverse problem with partial data}\label{partial}

In this section we give a proof of Theorem \ref{thm-partial}. Most of the argument of the previous section still holds. In particular, we still have the expansion \eqref{lambda-expansion}, which in the same way as above implies that
\begin{equation}\label{ii1}
\int_\dom(\epsilon^{(i)}-\epsilon^{(ii)})V^{\frac{1}{m}}W\dd x=\int_\dom(\lambda^{(i)}-\lambda^{(ii)})V^{\frac{q}{m}}W\dd x=0,
\end{equation}
for any $V$, $W$ which satisfy $\triangle V=\triangle W=0$, $V|_{\gam}=W|_{\gam}=0$, where  $\gam=\pr\dom\setminus\overline{\sig}$. In what follows we will focus on the integral identity for the $\epsilon$ coefficients, the case of $\lambda$ being nearly identical.

\begin{lem}
There exists $U_0\in C^\infty(\overline{\dom})$ such that $U_0\geq0$, $\triangle U_0=0$, $U_0|_{\gam}=0$, and $\gam\subset\supp(\pr_\nu U_0|_{\pr\dom})$.
\end{lem}
\begin{proof}
If $f\in C^\infty(\pr\dom)$ is such that $f\geq0$, $f|_{\gam}=0$, let $u_f\in C^\infty(\overline{\dom})$ be the solution of 
\begin{equation}
\left\{\begin{array}{l}\triangle u_f=0,\\[5pt] u_f|_{\pr\dom}=f.\end{array}\right.
\end{equation}
Suppose there exists a point $x_0\in\gam$ such that 
\begin{equation}\label{hyp}
\pr_\nu u_f(x_0)=0,\quad\forall f.
\end{equation}

Let $G_\dom(x,y)$ be the (Dirichlet) Green's function associated to the domain $\dom$. We have that $G_\dom(x,\cdot)\in C^\infty(\overline{\dom}\setminus\{x\})$  and 
\begin{equation}
u_f(x)=-\int_{\pr\dom} f(y)\nu(y)\cdot\nabla_y G_\dom(x,y)\dd S(y).
\end{equation}
We then get
\begin{equation}
\pr_\nu u_f(x_0)=-\int_{\pr\dom} f(y)\nu(x_0)\cdot\nabla_x\left(\nu(y)\cdot\nabla_y G_\dom(x_0,y)\right)\dd S(y)
\end{equation}
By \eqref{hyp} we conclude that
\begin{equation}
\nu(y)\cdot\nabla_y\left( \nu(x_0)\cdot\nabla_xG_\dom(x_0,y)\right)=0,\quad\forall y\in\sig.
\end{equation}
Note that we also have 
\begin{equation}
\left\{\begin{array}{l}\triangle_y\left( \nu(x_0)\cdot\nabla_xG_\dom(x_0,y)\right)=0,\\[5pt]
 \nu(x_0)\cdot\nabla_xG_\dom(x_0,y)=0,\quad\forall y\in\sig.\end{array}\right.
\end{equation}
By unique continuation 
\begin{equation}
\nu(x_0)\cdot\nabla_xG_\dom(x_0,y)=0,\quad\forall y\in\dom.
\end{equation}
This is a contradiction (e.g. since $\nu(x_0)\cdot\nabla_xG_\dom(x_0,y)\to\infty$ and $y\to x_0$). 

We  have shown that for any $x\in\gam$ there exists $f_x\in C^\infty(\pr\dom)$  such that $f_x\geq0$, $f_x|_{\gam}=0$, and $\pr_\nu u_{f_x}(x)<0$. As $\overline{\gam}$ is compact, we can find $f_1,\ldots,f_N\in C^\infty(\pr\dom)$  such that $f_j\geq0$, $f_j|_{\gam}=0$, for $j=1,\ldots,N$, and
\begin{equation}
\pr_\nu u_f|_{\overline{\gam}}<0,\quad f=\sum_{j=1}^N f_j.
\end{equation}
\end{proof}

If we choose $V(x)=U_0(x)+sU(x)$ in \eqref{ii1}, where $U\geq0$, $\triangle U=0$ and $U|_\gam=0$, then 
\begin{multline}
0=\left.\frac{\dd}{\dd s}\right|_{s=0}\int_\dom (\epsilon^{(i)}(x)-\epsilon^{(ii)}(x))(U_0(x)+sU(x))^{\frac{1}{m}}W(x)\dd x\\[5pt]
=\int_\dom(\epsilon^{(i)}(x)-\epsilon^{(ii)}(x)) U_0^{\frac{1}{m}-1}(x)U(x)W(x)\dd x.
\end{multline}
Note that it is possible to differentiate in $s$ since $U_0^{\frac{1}{m}-1}\in L^1(\dom)$. We can move from non-negative $U$, $W$ to complex valued ones in the obvious way using linearity. With the notation
\begin{equation}
F(x)=(\epsilon^{(i)}(x)-\epsilon^{(ii)}(x)) U_0^{\frac{1}{m}-1}(x)
\end{equation}
we then have
\begin{equation}\label{ii2}
\int_\dom F(x) U(x)W(x)\dd x=0
\end{equation}
for all functions $U$, $W$ harmonic in $\dom$ and such that $U|_\gam=W|_\gam=0$. A similar problem, with $F\in L^\infty(\dom)$, was considered in \cite{DKSU}. We will adapt their argument to our case where $F\in L^1(\dom)$.

It is observed in \cite[Section 3]{DKSU} that, without loss of generality, we may assume that
\begin{equation}
\dom\subset\{x\in\R^n: |x+e_1|<1\},\quad \gam=\{x\in\pr\dom: x_1\leq -2c\},
\end{equation}
where $c>0$ is a positive constant, and that $0\in\pr\dom$. This is proven by checking that the structure of the integral identity \eqref{ii2} remains unchanged under an appropriately chosen conformal transformation. The same holds here, and under the above mentioned transformation we still retain the properties that $F\in L^1(\dom)\cap C^\infty(\dom)$ and that $F(x)$ only becomes unbounded as $x\to\gam$. For any vector $z\in\C^n$ we will use the notation
\begin{equation}
z=(z_1,z'),\quad z'\in\C^{n-1}.
\end{equation}

\subsection{A local result}

Let $\zeta\in\C^n$ be such that $\zeta^2=0$ and let $\chi\in C_0^\infty(\R^n)$ be such that $\supp(\chi)\subset\{x\in\R^n:x_1<-c\}$ and $\chi|_\gam=1$. We can construct harmonic functions
\begin{equation}
U(x,\zeta,h)= e^{-\frac{i}{h}x\cdot\zeta}+R(x,\zeta,h),
\end{equation} 
where $R$ solves
\begin{equation}
\left\{\begin{array}{l}\triangle R=0,\\[5pt] R|_{\pr\dom}=-(e^{-\frac{i}{h}x\cdot\zeta}\chi)|_{\pr\dom},\end{array}\right.
\end{equation}
and $h$ is a positive parameter. In \cite[Section 2]{KKU} it is shown that there exists $C>0$ so that the remainder terms satisfy the estimate
\begin{equation}\label{r-estimate}
||R(\cdot,\zeta,h)||_{C(\overline{\dom})}\leq C\left(1+\frac{|\zeta|^k}{h^k}\right) e^{-\frac{c}{h}\im\zeta_1} e^{\frac{1}{h}|\im\zeta'|},
\end{equation}
whenever $\im\zeta_1\geq0$.

Let $\zeta,\eta\in C^n$ both be such that $\zeta^2=\eta^2=0$, $\im\zeta_1,\im\eta_1\geq0$. Choosing 
\begin{equation}
U(x)=U(x,\zeta,h), \quad W(x)=U(x,\eta,h)
\end{equation}
in \eqref{ii2} and using \eqref{r-estimate} to control the terms involving remainder terms we get that
\begin{multline}
\left|\int_\dom e^{-\frac{i}{h}x\cdot(\zeta+\eta)} F(x)\dd x\right|\\[5pt]
\leq C||F||_{L^1(\dom)}\left(1+\frac{|\zeta|^k}{h^k}\right)\left(1+\frac{|\eta|^k}{h^k}\right)
e^{-\frac{c}{h}\min(\im\zeta_1,\im\eta_1)}e^{\frac{2}{h}(|\im\zeta'|+|\im\eta'|)}.
\end{multline}

Let $\sigma=ie_1+e_2\in\C^n$ and suppose $a\in(0,\infty)$. If $z\in\C^n$ is such that $|z-2iae_1|<2\varepsilon$ then (see \cite[Section 2]{DKSU}) there exist $\zeta,\eta\in C^n$, with $\zeta^2=\eta^2=0$, such that
\begin{equation}
z=\zeta+\eta,\quad |\zeta-a\sigma|<Ca\varepsilon,\quad |\eta+a\overline{\sigma}|<Ca\varepsilon,
\end{equation}
if $\varepsilon>0$ is sufficiently small. It follows that for such a $z$ we have
\begin{equation}\label{3-8}
\left|\int_\dom e^{-\frac{i}{h}x\cdot z} F(x)\dd x\right|\leq C\left(\frac{a}{h}\right)^{2k}||f||_{L^1(\dom)} e^{-\frac{ca}{2h}}e^{\frac{2C\varepsilon a}{h}}.
\end{equation}

Let $\mathcal{T}$ denote the Segal-Bargmann transform
\begin{equation}
\mathcal{T} F(z)=\int_\dom e^{-\frac{1}{2h}(z-y)^2}F(y)\dd y,\quad z\in\C^n.
\end{equation}
Note that
\begin{equation}
(z-y)^2=(\re z-y)^2-(\im z)^2+2i(\re z-y)\cdot\im z,
\end{equation}
so
\begin{equation}
|\mathcal{T} F(z)|\leq e^{\frac{1}{2h}|\im z|^2}||F||_{L^1(\dom)}.
\end{equation}
When $\re z_1\geq 0$, since $\dom\subset\{x\in\R^n:x_1\leq0\}$, we have the sharper estimate
\begin{equation}
|\mathcal{T} F(z)|\leq e^{\frac{1}{2h}(|\im z|^2-|\re z_1|^2)}||F||_{L^1(\dom)}.
\end{equation}

It is observed in \cite[Section 4]{DKSU} that the Segal-Bargman transform can also be written as
\begin{equation}
\mathcal{T} F(z)=(2\pi h)^{-\frac{n}{2}}\int_{\R^n\times\dom} e^{-\frac{1}{2h}(z^2+t^2)}e^{-\frac{i}{h} y\cdot(t+iz)} F(y)\dd t\dd y.
\end{equation}
Then when $\re z_1\geq 0$,
\begin{multline}
|\mathcal{T}F(z)|\leq (2\pi h)^{-\frac{n}{2}}\int_{\R^n} e^{\frac{1}{2h}(|\im z|^2-|\re z|^2-t^2)}
\left|\int_\dom e^{-\frac{i}{h}y\cdot(t+iz)} F(y)\dd y\right|\dd t\\[5pt]
\leq (2\pi h)^{-\frac{n}{2}} e^{\frac{1}{2h}(|\im z|^2-|\re z|^2)}
\left(\int_{|t|\leq\varepsilon a} e^{-\frac{t^2}{2h}}\left|\int_\dom e^{-\frac{i}{h}y\cdot(t+iz)} F(y)\dd y\right|\dd t\right.\\[5pt]
+\left. \int_{|t|\geq\varepsilon a} e^{-\frac{t^2}{2h}}\left|\int_\dom e^{-\frac{i}{h}y\cdot(t+iz)} F(y)\dd y\right| \right)\\[5pt]
\leq e^{\frac{1}{2h}(|\im z|^2-|\re z|^2)}\left(\sup_{|t|\leq a\varepsilon}\left|\int_\dom e^{-\frac{i}{h}y\cdot(t+iz)}F(y)\dd y\right|\right.\\[5pt]
+\left.\sqrt{2} e^{\frac{1}{h}|\re z'|}e^{-\frac{\varepsilon^2a^2}{4h}}\int_\dom|F(y)|\dd y\right).
\end{multline}
When $|z-2ae_1|<\varepsilon a$, by \eqref{3-8}, the above gives
\begin{equation}
|\mathcal{T} F(z)|\leq
C\left(\frac{a}{h}\right)^{2k}||f||_{L^1(\dom)} e^{\frac{1}{2h}(|\im z|^2-|\re z|^2)}\left(e^{-\frac{ca}{2h}}e^{\frac{2C\varepsilon a}{h}}+e^{-\frac{\varepsilon^2a^2}{4h}} e^{\frac{\varepsilon a}{h}}\right).
\end{equation}
We are free to choose the parameters $\varepsilon$ and $a$, and with $\varepsilon$ sufficiently small and $a$ sufficiently large we have
\begin{equation}
|\mathcal{T} F(z)|\leq Ch^{-2k}||F||_{L^1(\dom)}e^{\frac{1}{2h}(|\im z|^2-|\re z|^2-\frac{ca}{2})}.
\end{equation}

Let
\begin{equation}
\Phi(z_1)=\left\{\begin{array}{ll}
|\im z_1|^2, \quad&\text{for }\re z_1\leq0\\ |\im z_1|^2-|\re z_1|^2,\quad&\text{for }\re z_1\geq0.
\end{array}\right.
\end{equation}
For $z_1\in\C$, $x'\in\R^{n-1}$, the above estimates can be combined into
\begin{multline}
e^{-\frac{1}{2h}\Phi(z_1)}|\mathcal{T}F(z_1,x')|\\[5pt]
\leq Ch^{-2k}||F||_{L^1(\dom)}\left\{\begin{array}{ll}
1,\quad&\text{for }z_1\in\C,\\[5pt] e^{-\frac{ca}{4h}},\quad&\text{for }|z_1-2a|\leq\frac{\varepsilon a}{2},|x'|<\frac{\varepsilon a}{2}.\end{array}\right.
\end{multline}
This is analogous to the estimate \cite[equation (4.7)]{DKSU}. As in the original proof we can now apply \cite[Lemma 4.1]{DKSU} to conclude that there exist $c',\delta>0$ such that
\begin{equation}
|\mathcal{T}F(x|\leq Ch^{-2k}||F||_{L^1(\dom)}e^{-\frac{c'}{2h}},
\end{equation}
whenever $x\in\dom\cap\{x\in\R^n:|x_1|\leq\delta\}$. For such an $x$ let $\varphi\in C_0^\infty(\dom)$ be such that $\varphi(y)=1$ for $y$ in a small neighborhood of $x$. Then we have that as $h\to0$
\begin{equation}
(2\pi h)^{-\frac{n}{2}}\int_\dom e^{-\frac{1}{2h}(x-y)^2}\varphi(y)F(y)\dd y\to \varphi(x)F(x)=F(x),
\end{equation}
and by the dominated convergence theorem
\begin{equation}
(2\pi h)^{-\frac{n}{2}}\int_\dom e^{-\frac{1}{2h}(x-y)^2}(1-\varphi(y))F(y)\dd y\to 0.
\end{equation}
It follows that $F(x)=0$ for any $x\in\dom\cap\{x\in\R^n:|x_1|\leq\delta\}$.

\subsection{The global result}

Suppose $\dom_1\subset\dom_2\subset\R^n$ are two bounded smooth domains such that $\dom_2\setminus\overline{\dom_1}$ is not empty. We further assume that $\pr\dom_1\cap\pr\dom_2$ is an open subset of $\pr\dom_1$ and it has $C^\infty$ boundary. The following lemma follows easily, via the Sobolev embedding theorem, from a result in \cite{KKU}.

\begin{lem}[see {\cite[Lemma 2.6]{KKU}}]
The space
\begin{equation}
\left\{\int_{\dom_2}G_{\dom_2}(\cdot,y)a(y)\dd y: a\in C^\infty(\overline{\dom_2}), \supp a\subset\dom_2\setminus\overline{\dom_1}\right\}
\end{equation}
is dense in 
\begin{equation}
\left\{ U\in C^\infty(\overline{\dom_1}): \triangle U=0, U|_{\pr\dom_1\cap\pr\dom_2}=0\right\},
\end{equation}
in the $L^\infty(\dom_1)$ topology.
\end{lem}

Let $x_1\in\dom$ and $x_0\in\sig$. Let $\theta:[0,1]\to\overline{\dom}$ be a $C^1$ curve connecting $x_0=\theta(0)$ to $x_1=\theta(1)$, such that $\theta((0,1])\subset\dom$ and $\theta'(0)$ is the interior normal to $\pr\dom$ at $x_0$. Let
\begin{equation}
\Theta_\varepsilon(t)=\{x\in\overline{\dom}:dist(x,\theta([0,t]))\leq\varepsilon\}, \quad t\in[0,1],
\end{equation}
and let
\begin{equation}
I=\left\{t\in[0,1]: F|_{\Theta_\varepsilon(t)\cap\dom}=0\right\}.
\end{equation}
Note that $I$ is a closed subset of $[0,1]$. If $\varepsilon>0$ is small enough, the previous subsection implies that $I$ is not empty.

Suppose $t\in I$. Clearly then $[0,t]\subset I$. For sufficiently small $\varepsilon$ we have that $\pr\Theta_\varepsilon(t)\cap\dom\subset\sig$. We can find a smooth domain $\dom_1$ such that 
\begin{equation}
\dom\setminus\Theta_\varepsilon(t)\subset\dom_1,\quad \gam\subset\pr\dom\cap\pr\dom_1,
\end{equation}
and a second smooth domain $\dom_2$ such that
\begin{equation}
\dom\subset\dom_2,\quad \pr\dom\cap\pr\dom_1\subset\pr\dom\cap\pr\dom_2.
\end{equation}
We extend $F$ to be zero on $\dom_1\setminus\overline\dom$, $\dom_2\setminus\overline\dom$.

Let $K:(\dom_2\setminus\overline{\dom_1})\times (\dom_2\setminus\overline{\dom_1})\to\R$ be
\begin{multline}
K(x,s)=\int_{\dom_1} F(y)G_{\dom_2}(x,y)G_{\dom_2}(s,y)\dd y\\[5pt]=\int_{\dom} F(y)G_{\dom_2}(x,y)G_{\dom_2}(s,y)\dd y.
\end{multline}
$K$ is harmonic in both $x$ and $s$. If $x,s\in\dom_2\setminus\overline\dom$, by \eqref{ii2}, we have that $K(x,s)=0$. By unique continuation it follows that $K(x,s)=0$ for all $x,s\in\dom_2\setminus\overline{\dom_1}$. If $a,b\in C^\infty(\overline{\dom_2})$, $\supp a,\supp b\subset\dom_2\setminus\overline{\dom_1}$, then
\begin{equation}
\int K(x,s)a(x)b(s)\dd x\dd s=0,
\end{equation}
and by the above approximation lemma we have that
\begin{equation}
\int_{\dom_1}F(y)U(y)W(y)\dd y=0,
\end{equation}
for all $U,W\in C^\infty(\overline{\dom_1})$ which are harmonic in $\dom_1$ and such that $U|_{\pr\dom_1\cap\pr\dom_2}=W|_{\pr\dom_1\cap\pr\dom_2}=0$.

Applying the result of the previous subsection to the domain $\dom_1$, we can conclude the there is a $t'>t$ such that $t'\in I$. This is enough to conclude that $I$ is open in $[0,1]$, and hence $I=[0,1]$. This gives that $F(x_1)=0$, which completes the proof of Theorem \ref{thm-partial}.

\paragraph{Acknowledgments:} C.C. was supported by NSF of China under grants 11931011 and 11971333. The research of T.G. is supported by the Collaborative Research Center, membership no. 1283, Universit\"at Bielefeld. 
GU was partly supported by
NSF, a Walker Family Professorship at UW, and a Si Yuan Professorship at
IAS, HKUST. He was also supported by a Simons Fellowship. Part of this
work was done while he was visiting IPAM in Fall 2021.

\bibliography{pme}
\bibliographystyle{plain}

\end{document}